\documentclass{amsart}%
\usepackage{amsfonts}
\usepackage{amsmath}
\usepackage{amssymb}
\usepackage{graphicx}
\usepackage{version}
\usepackage{bm}
\usepackage{epsfig}
\usepackage{caption}%
\usepackage{setspace}
\theoremstyle{plain}
\newtheorem{theorem}{Theorem}[section]

\newtheorem{proposition}{Proposition}[section]
\theoremstyle{definition}
\newtheorem{definition}{Definition}[section]

\newtheorem{remark}{Remark}[section]
\newtheorem*{notation*}{Notation}
\numberwithin{equation}{section}
\excludeversion{comment1}

\begin{document}
\title{Bicubic partially blended rational quartic surface}
\author[S.K. Katiyar]{S. K. Katiyar}
\address{Department of Mathematics, SRM Institute of Science and Technology, Chennai 603203, India}
\maketitle
\begin{abstract}
This paper investigates some univariate and bivariate constrained interpolation
problems using rational quartic fractal interpolation functions, which has been submitted long back in a reputed journal and revised as per the journal requirement. This research is extension of the work \cite{KCFractal2}.
\end{abstract}
\section{Introduction and Motivation}\label{ICAFWsec1}
\par 
Omnipresence of fractal functions is claimed by the fractal researchers (see, for instance, \cite{CK6,CNVS,CK7,NS3,N1,NVCSS})  in various contexts, inside the scope of pure mathematics and the real world applications. The problem of searching a sufficiently smooth function that preserves the
certain geometric shape properties of data is generally referred to as shape preserving interpolation, which is important in computer graphics and scientific visualization. The shape properties are mathematically expressed in terms of conditions such as positivity, monotonicity and convexity. Thus researchers keep trying to find best possible function that can interpolate the data with shape preserving property. As a submissive contribution to this goal, Chand and coworkers have initiated the study on shape preserving fractal interpolation and approximation using various families of polynomial and rational IFSs (see, for instance, \cite{AKBCSK1,CKIITR,KC,SKKKM,SKKthesis,KCGS,KCSJha,KCSole}). These shape preserving fractal interpolation schemes possess the novelty that the interpolants inherit the shape property in question and at the same time the suitable derivatives of these interpolants own irregularity in  finite or dense subsets of the interpolation interval. This attribute of shape preserving FIFs finds potential applications in various nonlinear phenomena.
\section{Preliminaries and Assumptions}\label{ICAFWsec1}
\noindent In this section, we present preliminaries, assumptions, and propositions required for proving the
main results. The details of the material presented here can be found in \cite{B1,B,PRM} for a complete overview.
\subsection{IFS for fractal functions}\label{ICAFWsec1a}
For $r\in \mathbb{N}$, let $\mathbb{N}_r$ denote the subset $\{1,2,\dots, r\}$ of $\mathbb{N}$.
Let a set of data points $\mathcal{D}=\{(x_i, y_i) \in \mathbb{R}^2: i \in \mathbb{N}_N\}$ satisfying $ x_1<x_2<\dots<x_N$, $N>2$, be given. Set $I = [x_1,x_N]$, $I_i = [x_i,x_{i+1}]$ for $i \in \mathbb{N}_{N-1}$. Suppose $L_n: I \rightarrow I_n$, $n \in \mathbb{N}_{N-1}$ be contraction homeomorphisms such that
\begin{equation}\label{ICAFWeq1a}
L_n(x_1)=x_n,\ L_n(x_N)=x_{n+1}.
\end{equation}
For instance, if $L_n(x) = a_n x+ b_n$, then the prescriptions in
(\ref{ICAFWeq1}) yield
\begin{equation}\label{ICAFWeq2a}
a_n=\frac{x_{n+1}-x_{n}}{x_N-x_1},~~b_n=\frac{x_N
x_n-x_1x_{n+1}}{x_N-x_1},~~ n \in \mathbb{N}_{N-1}.
\end{equation}
Let $0<r_n<1, n\in \mathbb{N}_{N-1}$, and $X:=I \times \mathbb{R}$. Let $N-1$ continuous mappings
 $F_n: X \to \mathbb{R}$ be given satisfying:
 \begin{equation}\label{ICAFWeq3a}
F_n(x_1,y_1)=y_n,\ \ F_n(x_N,y_N)=y_{n+1},~~ \vert
F_n(x,y)-F_n(x,y^*)\vert \leq r_n \vert
 y-y^*\vert,
\end{equation}
where $(x,y), (x,y^*)\in X$. Define functions $w_n: X \to I_n \times \mathbb{R},\ w_n(x,y)=\big(L_n(x),F_n(x,y)\big)$
$\forall~ i \in \mathbb{N}_{N-1} $. It is known \cite{B1} that there exists a metric on $\mathbb{R}^2$, equivalent to the Euclidean metric, with respect to which $w_n, n\in \mathbb{N}_{N-1}$, are contractions. The collection $\mathcal{I}=\{X; w_n:n \in \mathbb{N}_{N-1}\}$ is called an Iterated Function System (IFS). \noindent Associated with the IFS $\mathcal{I}$, there is a set
valued Hutchinson map $W:H(X) \rightarrow H(X)$ defined by $W(B)=\underset{n=1}  {\overset{N-1} \cup}w_n(B)$ for  $B\in H(X)$ , where $H(X)$ is the set of all nonempty compact subsets of $X$ endowed with the Hausdorff metric $h$. The Hausdorff metric h completes $H(X)$. Further, $W$ is a contraction map on the complete metric space $(H(X), h)$. By the Banach Fixed Point Theorem, there exists a unique set $G\in H(X)$ such that $W(G) = G$. This set $G$ is called the
attractor or deterministic fractal corresponding to the IFS $\mathcal{I}$. For any choices of $L_n$ and $F_n$
satisfying the conditions prescribed in (\ref{ICAFWeq1a})-(\ref{ICAFWeq3a}) , the following proposition holds.
\begin{proposition}\label{prop1}(Barnsley \cite{B1})
The IFS $\{X; w_n: n\in \mathbb{N}_{N-1}\}$ defined above admits a unique
attractor $G,$ and $G$ is the graph of a continuous function
$g:I\to \mathbb{R}$ which obeys $g(x_i)=y_i$\;for $i \in \mathbb{N}_N$.
\end{proposition}
\begin{definition}
The aforementioned function $g$  whose graph is the attractor of an IFS is called a \textbf{Fractal
Interpolation Function} (FIF) or a \textbf{self-referential function} corresponding to the IFS $\{X;\ w_n:n\in \mathbb{N}_{N-1}\}.$
\end{definition}
The above fractal interpolation function $g$ is obtained as the fixed point of the Read-Bajraktarevi\'{c} (RB) operator $T$ on a complete metric space $(\mathcal{G}, \rho)$ defined as
\begin{equation*}
(Th) (x)=  F_n\left(L_n^{-1}(x), h\circ L_n^{-1}(x)\right)\;
\forall~ x\in I_n ,\; n\in \mathbb{N}_{N-1},
\end{equation*}
where $\mathcal{G} :=\{h: I \rightarrow \mathbb{R}~|~ h~ \text{is continuous on}~ I,~ h(x_1)=y_1, h(x_N)=y_N\}$ equipped with the metric $\rho(h,h^*)=\max\{|h(x)-h^*(x)|: x \in I\}$ for $h,h^* \in \mathcal{G}$. It can be seen that $T$ is a contraction mapping on $(\mathcal{G}, \rho)$ with a contraction factor
$r^*:= \max\{r_n: n \in \mathbb{N}_{N-1}\}<1$. The fixed point of $T$ is the FIF $g$ corresponding to the IFS $\mathcal{I}$. Therefore, $g$ satisfies the functional equation:
\begin{equation}\label{ICAFWeq4aa}
g(x) = F_n\left(L_n^{-1}(x),g \circ L_n^{-1}(x)\right),\; x\in
I_n,\; n\in \mathbb{N}_{N-1},
\end{equation}
The most extensively studied FIFs in theory and applications  so
far are defined by the mappings:
\begin{equation}\label{ICAFWeq5a}
L_n(x)=a_nx+b_n,~  F_n(x,y)=\alpha_n y + R_n(x),~
n\in \mathbb{N}_{N-1}.
\end{equation}
Here $-1< \alpha_n<1$
and $R_n : I \to \mathbb{R}$ are suitable continuous functions
satisfying (\ref{ICAFWeq3a}).
The parameter $\alpha_n$ is called a scaling factor of the
transformation $w_n$, and $\alpha = (\alpha_1, \alpha_2, \dots,
\alpha_{N-1})$ is the scale vector corresponding to the IFS.

\subsection{$\alpha$-fractal function}\label{ICMCsec2b}
Barnsley and Navascués \cite{B1,N1} observed that the concept of FIFs can be used to generalize
a class of fractal functions associated with a given real-valued continuous function $f$ on a
compact interval $I$. Let $f \in \mathcal{C}(I)$ be a continuous function and consider the case:
\begin{equation}\label{ICMCeq5}
q_n(x)=f\circ L_n(x)-\alpha_n b(x),
\end{equation}
where $b :I \rightarrow \mathbb{R}$ is a continuous map that fulfills the conditions $b(x_1)=f(x_1)$, $b(x_N)=f(x_N)$, and $b\neq f$. Here the interpolation data are $\{(x_i, f(x_i)): i \in \mathbb{N}_N\}$. We define the $\alpha$-fractal function corresponding to $f$ in the following:
\begin{definition}
The continuous function
$f^\alpha:I \rightarrow \mathbb{R}$ whose graph is the attractor
of the IFS defined by (\ref{ICAFWeq5a})-(\ref{ICMCeq5}) is referred to as \textbf{$\alpha$-fractal function} (fractal approximation of $f$)
associated with f, with respect to ``base function" $b$, scale vector $\alpha$, and the partition $\mathcal{D}$.
\end{definition}
According to (\ref{ICAFWeq4aa}) , $f^\alpha$ satisfies the functional equation:
\begin{equation}\label{ICMCeq6}
f^\alpha(x)=f(x)+\alpha_n[(f^\alpha-b)\circ L_n^{-1}(x)]~\forall x\in
I_n,\; n\in \mathbb{N}_{N-1}.
\end{equation}
The $\alpha$-fractal function $f^\alpha$ obtained by perturbing a given continuous function  $f \in \mathcal{C}(I)$ with the help of a finite sequence of base functions $B:=\{b_n \in \mathcal{C}(I), ~b_n(x_1)=f(x_1),~ b_n(x_N)=f(x_N),~ b_n\neq f\}$ instead of a single base function is more advantageous. For instance, in generalizing rational splines with different shape parameters in different subintervals determined by interpolation points. That is, consider
\begin{equation}\label{ICMCeq7}
q_n(x)=f\circ L_n(x)-\alpha_n b_n(x).
\end{equation}
According to (\ref{ICAFWeq4aa}) , $f^\alpha$ satisfies the functional equation:
\begin{equation}\label{ICMCeq8}
f^\alpha(x)=f(x)+\alpha_n[(f^\alpha-b_n)\circ L_n^{-1}(x)]~\forall~ x\in
I_n,\; n\in \mathbb{N}_{N-1}.
\end{equation}
Note that for $\alpha=0$, $f^\alpha = f$. Thus aforementioned equation may be treated as an entire family of functions $f^\alpha$ with $f$ as its germ. By this method one can define fractal analogues of any continuous function.
\subsection{Smooth $\alpha$-Fractal Functions}\label{ICMCsec2d}
To construct $\alpha$-fractal function $f^\alpha$, the procedure can be easily carried out, which is the content of the following proposition. Details may be consulted in \cite{KC}. Let us introduce the following notation for a continuous function $g$ defined on a compact interval $J$ :
\begin{align*}
m(g;J)= min~\{g(x): x\in J\},~~~~~~M(g;J)= max~\{g(x): x\in J\}.
\end{align*}
\begin{proposition}\label{prop3}
Let $f \in \mathcal{C}^r(I)$ be such that $M_1\le f^{(r)}(x)\le M_2$ for all $x\in I$ and for some suitable constant $M_1$ and $M_2$. The $\alpha$-fractal function $f^\alpha$ \big(cf. (\ref{ICMCeq6})\big) corresponding to $f$ satisfies $M_1\le (f^\alpha)^{(r)}(x)\le M_2$ for all $x\in I$, provided the finite sequence of base functions  satisfying $B=\{b_n \in \mathcal{C}^r(I), ~b_n^{(k)}(x_1)=f^{(k)}(x_1),~ b_n^{(k)}(x_N)=f_n^{(k)}(x_N),~ b_n\neq f, k=0,1,\dots,r\}$ and the scaling factors $|\alpha_n|<a_n^{r}$ for all $n \in \mathbb{N}_{N-1}$ obeying the following additional conditions are selected.
\begin{align*}
max\left\{\dfrac{a_n^{r}[M_1-m(f^{(r)};I_n)]}{M_2-m(b_n^{(r)};I)},
   \dfrac{-a_n^{r}[M_2-M(f^{(r)};I_n)]}{M(b_n^{(r)};I)-M_1}\right\} \le \alpha_n\nonumber\\
\le min\left\{\dfrac{a_n^{r}[m(f^{(r)};I_n)-M_1]}{M(b_n^{(r)};I)-M_1},
   \dfrac{a_n^{r}[M_2-M(f^{(r)};I_n)]}{M_2-m(b_n^{(r)};I)}\right\}
\end{align*}
\end{proposition}

\section{Construction of $\alpha$-Fractal Rational Quartic Spline with Shape Parameter}\label{ICAFWsec2}
\noindent In this section, we construct $\alpha$-fractal rational quartic spline with quartic numerator and linear denominator with one family of shape parameter \cite{KCFractal2}, which generalize a class of classical rational quartic interpolants described in \cite{ABD}.\\
Let us first construct $\alpha$-fractal rational quartic spline based on function values and derivative values. For this, consider a set of Hermite data $\mathcal{D}=\{(x_n, y_n,d_n) \in \mathbb{R}^3: n \in \mathbb{N}_N\}$ satisfying $ x_1<x_2<\dots<x_N$. A traditional nonrecursive rational quartic
spline $Q \in \mathcal{C}^1(I)$ associated with $\mathcal{D}$ is defined in a piecewise manner (see \cite{ABD} for details). For
$\theta:=\frac{x-x_1}{x_N-x_1}$ and $h_n=x_{n+1}-x_n$, $ x \in I$,
\begin{equation} \label{ICAFWeq1}
Q\big(L_n(x)\big)= \frac{A_n^\diamond(1-\theta)^4 +B_n^\diamond(1-\theta)^3 \theta
+ C_n^\diamond(1-\theta)^2 \theta^2  +D_n^\diamond(1-\theta) \theta^3+ E_n^\diamond \theta^4}{\lambda_n(1-\theta)+ \theta},
\end{equation}
where $A_n^\diamond = \lambda_n y_n$,~~ $B_n^\diamond = (3\lambda_n+1)y_n+\lambda_n h_n d_n$,~~ $C_n^\diamond = 3\lambda_ny_{n+1}+3y_n$,~~ $D_n^\diamond = (\lambda_n+3)y_{n+1}-h_nd_{n+1}$,~~ $E_n^\diamond =y_{n+1}$ and $\lambda_n>0$ are free shape (tension) parameter.
The rational interpolant $Q$ satisfies the Hermite interpolation conditions, $Q(x_n)= y_n$ and $Q^{(1)}(x_n)=d_n,$ for $n\in\mathbb{N}_N$, where the values of $d_n$ are estimated or given.
\\To develop the $\alpha$-fractal rational quartic spline
corresponding to $Q$, assume $|\alpha_n| \le a_n,$ and select a family $B=\{b_n \in \mathcal{C}^1(I): n
\in \mathbb{N}_N\}$ satisfying the conditions $b_n{(x_1)} = Q(x_1)= y_1$,
$b_n{(x_N)}= Q(x_N)=y_N$, $b_n^{(1)}{(x_1)}=Q^{(1)}(x_1)=d_1$, and
$b_n^{(1)}{(x_N)}=Q^{(1)}{(x_N)}=d_N$. There are variety of choices for $B$. For our convenience, we take $b_n$ to be a rational function of similar form as that of the classical interpolant $Q$. For $n\in \mathbb{N}_{N-1}$, $x \in I$, and $\theta:=
\frac{x-x_1}{x_N-x_1}$, our choice for $b_n$ is
\begin{equation}\label{ICAFWeq2}
b_n (x)= \frac{A_n^\triangleleft(1-\theta)^4 +B_n^\triangleleft(1-\theta)^3 \theta
+ C_n^\triangleleft(1-\theta)^2 \theta^2  +D_n^\triangleleft(1-\theta) \theta^3+ E_n^\triangleleft\theta^4}{\lambda_n(1-\theta)+ \theta},
\end{equation}
where the coefficients $A_n^\triangleleft$, $B_n^\triangleleft$, $C_n^\triangleleft$, $D_n^\triangleleft$, and $E_n^\triangleleft$
are determined through the conditions $b_n(x_1)=y_1$,
$b_n(x_N)=y_N$, $b_n^{(1)}(x_1) = d_1$, $b_n^{(1)}(x_N)= d_N$.
After applying these conditions on $b_n$, we can easily get
$A_n^\triangleleft =  \lambda_n y_1,\;\; B_n^\triangleleft = (3\lambda_n+1) y_1+(x_N-x_1) \lambda_n d_1,
\;\; C_n^\triangleleft = 3\lambda_ny_N+3y_1,\;\;D_n^\triangleleft = (\lambda_n+3)y_N-(x_N-x_1) d_N,\;\; E_n^\triangleleft =y_N.$
Consider the $\alpha$-fractal rational quartic spline corresponding
to $Q$ as
\begin{equation}\label{ICAFWeq3}
Q^\alpha \big(L_n (x)\big)= \alpha_n Q^\alpha(x) +
Q\big(L_n(x)\big)-\alpha_n b_n(x),~ x\in I,~ n\in \mathbb{N}_{N-1}.
\end{equation}
Using (\ref{ICAFWeq1}) and (\ref{ICAFWeq2}) in (\ref{ICAFWeq3}), we have
\begin{eqnarray}\label{ICAFWeq4}
Q^\alpha\big(L_n (x)\big)= \alpha_n Q^\alpha(x)+
\frac{P_n(x)}{Q_n(x)},
\end{eqnarray}
\begin{eqnarray*}
P_n(x) &=& A_n (1-\theta)^4 +B_n(1-\theta)^3 \theta+C_n (1-\theta)^2 \theta^2  +D_n (1-\theta) \theta^3+E_n \theta^4\nonumber\\
Q_n(x)&=&~\lambda_n (1-\theta)+\theta,~ n\in \mathbb{N}_{N-1},~ \theta=
\frac{x-x_1}{x_N-x_1}\nonumber,
\end{eqnarray*}
where,
$A_n=\lambda_n(y_n-\alpha_n y_1), B_n= \lambda_n\{h_nd_n-\alpha_n (x_N-x_1)d_1\}+(3\lambda_n+1)(y_n-\alpha_n y_1), C_n= 3(y_n-\alpha_n y_1)+3\lambda_n(y_{n+1}-\alpha_n y_N), D_n=(\lambda_n+3)(y_{n+1}-\alpha_ny_N)-\{h_nd_{n+1}-\alpha_n(x_N-x_1) d_N\}, E_n=(y_{n+1}-\alpha_n y_N)$.\\
\begin{proposition}\label{prop4}
Let $Q \in \mathcal{C}^r(I)$. Suppose $\mathcal{D}=\{x_1,x_2,\dots,x_N\}$ be an arbitrary partition on $I$ satisfying $ x_1<x_2<\dots<x_N$. Let $|\alpha_n|<a_n^{r}$ for all $n \in \mathbb{N}_{N-1}$. Further suppose that $B=\{b_n\in \mathcal{C}^r(I):n \in \mathbb{N}_{N-1}\}$ fulfills $b_n^{(k)}(x_1)= Q^{(k)}(x_1)$, $b_n^{(k)}(x_N)= Q^{(k)}(x_N)$ for $k=0,1,\dots, r$. Then the corresponding fractal function $Q^\alpha$ is $r$-smooth, and $(Q^\alpha)^{(k)}(x_n)= Q^{(k)}(x_n)$ for  $n\in \mathbb{N}_N$ and $k=0,1,\dots, r$.
\end{proposition}
\begin{remark}\label{ICAFWrem1}
For all $n \in \mathbb{N}_{N-1}$, if $\alpha_n=0$, then the $\alpha$-fractal rational quartic spline (\ref{ICAFWeq4}) recovers the classical rational quartic interpolant $Q$ constructed in \cite{ABD}.
\end{remark}
\begin{remark}\label{ICAFWrem2}
If we take  $\lambda_n=\frac{u_n}{v_n}$ for all $n \in \mathbb{N}_{N-1}$, with $u_n$, $v_n$ are non-zero shape parameters such that sign($u_n$)=~sign($v_n$), then ${Q^*}^\alpha\in \mathcal{C}^1[x_1,x_N]$ is the $\alpha$-fractal rational quartic spline with family of two shape parameters $u_n$ and $v_n$,
\begin{eqnarray}\label{ICAFWeq4a}
{Q^*}^\alpha\big(L_n (x)\big)= \alpha_n {Q^*}^\alpha(x)+
\frac{P_n(x)}{Q_n(x)},
\end{eqnarray}
The value of new $P_n(x)$ and $Q_n(x)$ are given as follows:\\
$P_n(x)=~u_n(y_n-\alpha_n y_1) (1-\theta)^4 + [u_n\{h_nd_n-\alpha_n (x_N-x_1)d_1\}+(3u_n+v_n)(y_n-\alpha_n y_1)] (1-\theta)^3  \theta+ [3v_n(y_n-\alpha_n y_1)+3u_n(y_{n+1}-\alpha_n y_N)] (1-\theta)^2 \theta^2  +[(u_n+3v_n)(y_{n+1}-\alpha_ny_N)-v_n\{h_nd_{n+1}-\alpha_n(x_N-x_1) d_N\}] (1-\theta) \theta^3 + v_n(y_{n+1}-\alpha_n y_N) \theta^4,$
~~$Q_n(x)= u_n (1-\theta)+v_n \theta.$
\end{remark}
\noindent For all $n \in \mathbb{N}_{N-1}$, if $\alpha_n=0$ and $\lambda_n=\frac{u_n}{v_n}$, where $u_n$, $v_n$ are non-zero shape parameters such that sign($u_n$)=~sign($v_n$), then the $\alpha$-fractal rational quartic spline recovers the classical rational quartic interpolant $Q^*$ constructed in \cite{WQ}.
\section{Convergence Analysis for $\alpha$-Fractal Rational Quartic Spline}\label{ICAFWsec3}
For $r\in \mathbb{N}$, let $\mathbb{N}_r$ denote the subset $\{1,2,\dots, r\}$ of $\mathbb{N}$.
$\Gamma_n := \max\{\lambda_n,1\}, \xi_n^* := \max\{u_n,v_n\}, \mu_n:= \min\{\lambda_n,1\}, \mu_n^*:= \min\{u_n,v_n\}, k_n :=\max\{|\Phi^{(1)}(x_n)-d_n|,|\Phi^{(1)}(x_{n+1})-d_{n+1}|\}, c_n:=|d_n|+|d_{n+1}|, \Gamma:=\max\{\Gamma_n: n\in \mathbb{N}_{N-1}\}, \mu:= \min\{\mu_n: n\in \mathbb{N}_{N-1}\}, k:=\max\{k_n: n\in \mathbb{N}_{N-1}\},
c:=\max\{c_n: n\in \mathbb{N}_{N-1}\}, h_n:=x_{n+1}-x_n, \Delta_n:=\frac{y_{n+1}-y_n}{h_n}~\text{for}~ n\in \mathbb{N}_{N-1},  h :=\max\{h_n: n\in \mathbb{N}_{N-1}\}$.
\begin{theorem}\label{ICAFWthm1}\cite{KCFractal2}
Let $Q^\alpha$ be the $\alpha$-fractal rational quartic spline for the original function $\Psi\in\mathcal{C}^4[x_1,x_N]$  with respect to the data $\{(x_n, y_n, d_n) : n\in \mathbb{N}_N \}$. Then,\\
\begin{equation*}
\|\Psi-Q^\alpha \|_{\infty}\le\frac{h^4}{384}\|\Phi^{(4)}\|_{\infty}+\frac{kh}{4}+\frac{h\xi c}{16\mu}+\frac {|\alpha|_{\infty}}{ 1 -|\alpha|_{\infty}} (\|y \|_{\infty}+\frac{1}{4}h\|d \|_{\infty}+  K_0).
\end{equation*}
\end{theorem}
\subsection{Containment of the Graph of the $\alpha$-Fractal Rational Quartic Spline within a Rectangle}\label{r2pps52}
Given data set $\{(x_i,y_i):i\in \mathbb{N}_N\}$, the problem is to
find conditions on the scaling factors and shape parameters of the rational quartic IFS so that the graph $G(Q^\alpha)=\{(x,Q^\alpha(x)): x \in I\}$ of the corresponding $\alpha$-fractal rational quartic spline $Q^\alpha$ lies
within the prescribed rectangle $R= I \times [c,d]$, where $c< \min\{y_i: i\in
\mathbb{N}_N\}$ and $d> \max\{y_i: i\in \mathbb{N}_N\}$. Graph $G$ is the attractor of the IFS
$\mathcal{I}=\Big\{I\times\mathbb{R};\big(L_n(x), F_n(x,y)=\alpha_n y+ q_n(x)\big):
n\in \mathbb{N}_{N-1}\Big\}$.  The FIF $g$ is obtained by iterating the IFS
$\mathcal{I}$ and $I=[x_1, x_N]$ is the attractor of the IFS $\{\mathbb{R};
L_n(x): n\in \mathbb{N}_{N-1}\}$. Hence, for $G$ to lie within $R$, it suffices
to prove that $ \alpha_n y+ q_n(x) \in [c,d]~ \forall~ (x,y) \in
R$.\\\\
\noindent\textbf{Case-I}  We assume $0 \le \alpha_n <a_n$ for all $n\in
\mathbb{N}_{N-1}$. Let $(x,y) \in R$. Then, with our assumption on $\alpha_n$,
we have $\alpha_n c \le \alpha_n y \le \alpha_n d $. This implies
$\alpha_n c+ \frac{P_n(x)}{Q_n(x)} \le \alpha_n y+ \frac{P_n(x)}{Q_n(x)} \le \alpha_n d + \frac{P_n(x)}{Q_n(x)}$. Consequently, for $G$ to lie within $R$, it suffices to
have the following conditions for all $n\in \mathbb{N}_{N-1}:$
\begin{equation}\label{ICAFWeq26a}
c \le \alpha_n c + \frac{P_n(x)}{Q_n(x)},
\end{equation}
\begin{equation}\label{ICAFWeq26b}
 \alpha_n d + \frac{P_n(x)}{Q_n(x)} \le d.
\end{equation}
Now, (\ref{ICAFWeq26a}) is satisfied if
\begin{equation}\label{ICAFWeq26c}
c(1-\alpha_n) \le \frac{A_n (1-\theta)^4 +B_n(1-\theta)^3 \theta+C_n (1-\theta)^2 \theta^2  +D_n (1-\theta) \theta^3+E_n \theta^4}{\lambda_n (1-\theta)+\theta},
\end{equation}
where the constants $A_n, B_n, C_n$, $D_n$, and $E_n$ are
given in (\ref{ICAFWeq4}). Note that the expression for
$Q_n(x)$ can be written in the degree elevated form as follows:
\begin{equation}\label{ICAFWeq26d}
\begin{split}
\lambda_n (1-\theta)+\theta \equiv~& \lambda_n (1-\theta)^3 + (2\lambda_n+1)(1-\theta)^2 \theta + (\lambda_n+2) (1-\theta) \theta^2+  \theta^3\\ \equiv~& \lambda_n (1-\theta)^4+(3\lambda_n+1)(1-\theta)^3 \theta+
(3\lambda_n+3)(1-\theta)^2 \theta^2+\\~&(\lambda_n+3) (1-\theta) \theta^3+\theta^4.
\end{split}
\end{equation}
Using (\ref{ICAFWeq26d}) in (\ref{ICAFWeq26c}), we have
\begin{equation}\label{ICAFWeq26e}
\begin{split}
&\big[A_n -c (1-\alpha_n)\lambda_n\big](1-\theta)^4 + \big[B_n-c
(1-\alpha_n)(3\lambda_n+1)\big]
(1-\theta)^3\theta+\big[ C_n-\\&c (1-\alpha_n)(3\lambda_n+3)\big]
(1-\theta)^2 \theta^2+ \big[ D_n-c(1-\alpha_n) (\lambda_n+3) \big] (1-\theta)\theta^3+\\&\big[ E_n-c(1-\alpha_n) \big]\theta^4
\ge 0.
\end{split}
\end{equation}
With the substitution $\theta=\frac{\nu}{\nu+1}$,  (\ref{ICAFWeq26e}) is equivalent to\\
\begin{equation}\label{ICAFWeq26f}
\begin{split}
&\big[E_n-c(1-\alpha_n) \big]\nu^4+\big[ D_n-c(1-\alpha_n) (\lambda_n+3) \big]\nu^3 +\big[ C_n-c (1-\alpha_n)(3\lambda_n+3)\big]\cdot\\& \nu^2 + \big[B_n-c
(1-\alpha_n)(3\lambda_n+1)\big]\nu+ \big[A_n -c (1-\alpha_n)\lambda_n\big]\ge0~~~ \forall ~\nu\ge0.
\end{split}
\end{equation}
It can be seen that the polynomial in (\ref{ICAFWeq26e}) is
positive if the following system of inequalities hold:
\begin{equation*}
\begin{split}
&E_n-c(1-\alpha_n) \ge 0, D_n-c(1-\alpha_n) (\lambda_n+3)\ge 0, C_n-c (1-\alpha_n)(3\lambda_n+3) \ge 0,\\
&B_n-c
(1-\alpha_n)(3\lambda_n+1)
\ge 0, A_n -c (1-\alpha_n)\lambda_n\ge 0.
\end{split}
\end{equation*}
Since $\lambda_n>0$, the selection of $\alpha_n$ satisfying
 $\alpha_n \le
\dfrac{y_{n+1}-c}{y_N-c}$ and $\alpha_n \le \dfrac{y_n-c}{y_1-c}$ ensures first and fifth inequality. Let
$\alpha_n \le \min \big\{\dfrac{y_n-c}{y_1-c},
\dfrac{y_{n+1}-c}{y_N-c}\big\}$. By this choice of $\alpha_n$, $C_n-c (1-\alpha_n)(3\lambda_n+3) \ge 0$ is obviously true because $C_n-c (1-\alpha_n)(3\lambda_n+3)\equiv 3[(y_n-\alpha_ny_1)-c (1-\alpha_n)]+3\lambda_n[(y_{n+1}-\alpha_ny_N)-c (1-\alpha_n)]\ge 0$.
Assume,  $\alpha_n \le \dfrac{h_nd_n}{(x_N-x_1)d_1}$, then the condition $B_n-c(1-\alpha_n)(3\lambda_n+1)\ge 0$ is met if $\lambda_n \ge \dfrac{-[(y_n-\alpha_ny_1)-c (1-\alpha_n)]}{3[(y_n-\alpha_ny_1)-c (1-\alpha_n)]+h_nd_n-\alpha_n(x_N-x_1)d_1}$. The condition $D_n-c(1-\alpha_n) (\lambda_n+3)\ge 0$ is met if $\lambda_n \ge -3+\dfrac{h_nd_{n+1}-\alpha_n(x_N-x_1)d_N}{(y_{n+1}-\alpha_ny_N)-c (1-\alpha_n)}.$ Hence, the following conditions are sufficient to verify (\ref{ICAFWeq26a}) if scaling factors and shape parameters chosen as
\begin{equation}\label{ICAFWeq26g}\left.
\begin{split}
0 ~&\le \alpha_n  < \min\Big\{\dfrac{y_n-c}{y_1-c},
\dfrac{y_{n+1}-c}{y_N-c},\dfrac{h_nd_n}{(x_N-x_1)d_1}\Big\},\\
\lambda_n~& \ge \lambda_{1n}:=~ \dfrac{-[(y_n-c)-\alpha_n(y_1-c) ]}{3[(y_n-c)-\alpha_n(y_1-c)]+h_nd_n-\alpha_n(x_N-x_1)d_1},\\
\lambda_n~& \ge \lambda_{2n}:=~ -3+\dfrac{h_nd_{n+1}-\alpha_n(x_N-x_1)d_N}{(y_{n+1}-c)-\alpha_n(y_N-c)}.
\end{split}\right\}
\end{equation}
Substituting the expression for $Q_n(x)$ in (\ref{ICAFWeq26d}), (\ref{ICAFWeq26b}) reduces to\\
$\big[d (1-\alpha_n)\lambda_n-A_n\big](1-\theta)^4 + \big[d
(1-\alpha_n)(3\lambda_n+1)-B_n\big]
(1-\theta)^3\theta+\big[d (1-\alpha_n)(3\lambda_n+3)-C_n\big]
(1-\theta)^2\theta^2+ \big[d(1-\alpha_n)(\lambda_n+3)-D_n \big](1-\theta) \theta^3+ \big[d(1-\alpha_n)-E_n \big]\theta^4
\ge 0$.\\
Now, with the substitution $\theta=\frac{\nu}{\nu+1}$,  we may rewrite it as
$\big[d(1-\alpha_n)-E_n \big]\nu^4+\big[d(1-\alpha_n)(\lambda_n+3)-D_n \big]\nu^3 +\big[d (1-\alpha_n)(3\lambda_n+3)-C_n\big] \nu^2 + \big[d
(1-\alpha_n)(3\lambda_n+1)-B_n\big] \nu+ \big[d (1-\alpha_n)\lambda_n-A_n\big]\ge0~\forall~\nu\ge0$. We proceed as above and the following conditions verify (\ref{ICAFWeq26b}):
\begin{equation}\label{ICAFWeq26h}\left.
\begin{split}
0 ~&\le \alpha_n < \min\Big\{\frac{d-y_n}{d-y_1},
\frac{d-y_{n+1}}{d-y_N},\frac{3(d-y_n)-h_nd_n}{3(d-y_1)-(x_N-x_1)d_1} \Big\},\\
\lambda_n~& \ge \lambda_{3n}:=~ \dfrac{-[(d-y_n)-\alpha_n(d-y_1)]}{3[(d-y_n)-\alpha_n(d-y_1)]-
[h_nd_n-\alpha_n(x_N-x_1)d_1]},\\
\lambda_n~& \ge \lambda_{4n}:=~ -3+\dfrac{\alpha_n(x_N-x_1)d_N-h_nd_{n+1}}{(d-y_{n+1})-\alpha_n(d-y_N)}.
\end{split}\right\}
\end{equation}
Let us denote $\alpha_n^{max}=\min\Big\{a_n,\dfrac{y_n-c}{y_1-c},
\dfrac{y_{n+1}-c}{y_N-c},\dfrac{h_nd_n}{(x_N-x_1)d_1},\dfrac{d-y_n}{d-y_1},
\dfrac{d-y_{n+1}}{d-y_N},\\~~~~~~~~~~~~~~~~~~~~~~~~~~~~
~~~~~~~~~~~~~~~~~~~~~\dfrac{3(d-y_n)-h_nd_n}{3(d-y_1)-(x_N-x_1)d_1}\Big\}$.
\\\\
\noindent\textbf{Case-II} We consider $-a_n< \alpha_n <0$. Let $(x,y) \in R$. In this case (\ref{ICAFWeq26a}) and (\ref{ICAFWeq26b}) will be replaced respectively by
\begin{equation}\label{ICAFWeq26i}
c \le \alpha_n d+ q_n(x) ~~\text{and}~~  \alpha_n c + q_n(x) \le
d.
\end{equation}
To make our considerations concise, we avoid the computational details that yield the following conditions for scaling factors and shape parameters
\begin{equation}\label{ICAFWeq26j}\left.
\begin{split}
\alpha_n ~&> \alpha_n^{min}=\max\big\{-a_n, \dfrac{y_n-c}{y_1-d},
\dfrac{y_{n+1}-c}{y_N-d},\dfrac{-h_nd_n-3
(y_n-c)}{3(d-y_1)-(x_N-x_1)d_1},\dfrac{d-y_n}{c-y_1},\\&
~~~~~~~~~~~~~~~~~~~~~~~~~~~~~~~\dfrac{d-y_{n+1}}{c-y_N}, \dfrac{h_nd_n-3(d-y_n)}{3(y_1+c)+(x_N-x_1)d_1}\big\}, \\
\lambda_n~& \ge \lambda_{5n}:=~ \dfrac{-[\alpha_n(d-y_1)-(c-y_n)]}{3[\alpha_n(d-y_1)-(c-y_n)]+
[h_nd_n-\alpha_n(x_N-x_1)d_1]},\\
\lambda_n~& \ge \lambda_{6n}:=~ -3+\dfrac{h_nd_{n+1}-\alpha_n(x_N-x_1)d_N}{\alpha_n(d-y_N)-(c-y_{n+1})},\\
\lambda_n~& \ge \lambda_{7n}:=~ \dfrac{-[(d-y_n)-\alpha_n(c-y_1)]}{3[(d-y_n)-\alpha_n(c-y_1)]-
[h_nd_n-\alpha_n(x_N-x_1)d_1]},\\
\lambda_n~& \ge \lambda_{8n}:=~ -3+\dfrac{\alpha_n(x_N-x_1)d_N-h_nd_{n+1}}{(d-y_{n+1})-\alpha_n(c-y_N)},
\end{split}\right\}
\end{equation}
Based on the above discussion, we state the following theorem.
\begin{theorem}\label{ICMCthm5}
Suppose a data set $\{(x_i,y_i):i\in \mathbb{N}_N\}$ is given, and $Q^\alpha$ is the corresponding $\alpha$-fractal rational quartic spline  described in
(\ref{ICAFWeq4}). Then the following conditions on the
scaling factors and the shape parameters in each subinterval are
sufficient for the  containment of the graph of $Q^\alpha$ in the
rectangle $R= I \times [c,d]:$
\begin{equation*}
\begin{split}
 \alpha_n ^{min} &< \alpha_n  < \alpha_n ^{max},~\lambda_n >\max\{0, \lambda_{1n}, \lambda_{2n}, \lambda_{3n},
 \lambda_{4n}, \lambda_{5n}, \lambda_{6n}, \lambda_{7n},
 \lambda_{8n}\}.
\end{split}
\end{equation*}
\end{theorem}
\subsection{$\alpha$-Fractal Rational Quartic Spline Above the Line}
\begin{theorem}
Suppose a data set $\{(x_i,y_i):i\in \mathbb{N}_N\}$ lies above the line $t=mx+k$, that is, $y_i>t_i=t(x_i)$ for $i\in \mathbb{N}_N$. The graph of the corresponding $\alpha$-fractal rational quartic spline $Q^\alpha$ lies above the line $t=mx+k$ if the scaling factors $|\alpha_n|<a_n$ and the shape parameters $\lambda_n >0, n\in \mathbb{N}_{N-1},$ further satisfy
\begin{equation*}
\begin{split}
0\le \alpha_n < \min \big\{\dfrac{y_n-t_n}{y_1-t_1},
\dfrac{y_{n+1}-t_{n+1}}{y_N-t_N}\big\} ,~\lambda_n > \max \{0,\lambda_{9n}, \lambda_{10n}\},
\end{split}
\end{equation*}
where
$\lambda_n \ge \lambda_{9n}:=\dfrac{-[2\{(y_n-t_n)-\alpha_n(y_1-t_1)\}+(y_n-t_{n+1})-\alpha_n(y_1-t_N)]}
{[(y_{n+1}-t_n)-\alpha_n(y_N-t_1)+2\{(y_{n+1}-t_{n+1})-\alpha_n(y_N-t_N)\}]}$.
$\lambda_n \ge \lambda_{10n}:= -\dfrac{y_n-t_n-\alpha_n(y_1-t_1)}{l_n},\\ \lambda_n \ge
\lambda_{11n}:= -\dfrac{m_n}{[y_{n+1}-t_{n+1}-\alpha_n(y_N-t_N)]}$,
$l_n=[2(y_n-t_n)+(y_n-t_{n+1})+h_n d_n-\alpha_n\{2(y_1-t_1)+(y_1-t_N)+(x_N-x_1)d_1\}]$,
$m_n=[(y_{n+1}-t_n)+2(y_{n+1}-t_{n+1})-h_n d_{n+1}-\alpha_n\{(y_N-t_1)+2(y_N-t_N)
-(x_N-x_1)d_N\}]$.
\end{theorem}
\begin{proof}
Since $Q^\alpha(x_i) =y_i> t_i$ for all $i\in \mathbb{N}_N$, therefore to prove $Q^\alpha(\tau)> m \tau+k$ for all $\tau\in I$, it is sufficient to verify that $Q^\alpha(x)> m x+k, x\in I$ implies $Q^\alpha(L_n(x)) > m L_n(x)+k$ for all $n\in \mathbb{N}_{N-1}$. Assume $Q^\alpha(x)> m x+k$ for some $x$. We need to make sure that
\begin{equation}\label{ICAFWeq26k}
\alpha_n Q(x)+\frac{P_n(\theta)}{Q_n(\theta)}>m(a_nx+b_n)+k.
\end{equation}
We shall impel the scaling parameters such that $0\le\alpha_n<a_n$ for all $n\in \mathbb{N}_{N-1}$. Since $Q_n(\theta)>0$ and keeping in mind the assumptions $Q^\alpha(x)> m x+k$, cross multiplying and rearranging it can be observed that (\ref{ICAFWeq26k}) holds if
\begin{equation}\label{ICAFWeq26l}
\alpha_n(m x+k)Q_n(\theta)+ P_n(\theta)- (ma_nx+mb_n+k)Q_n(\theta)>0, \theta\in[0,1].
\end{equation}
Substituting $x=x_1+\theta(x_N-x_1)$, the expression for $P_n(\theta)$ and using the degree elevated form of $Q_n(\theta)$ from (\ref{ICAFWeq26d}), (\ref{ICAFWeq26l}) reduces to
\begin{equation}\label{ICAFWeq26m}
A_n^*(1-\theta)^4 +B_n^*(1-\theta)^3 \theta
+ C_n^*(1-\theta)^2\theta^2 + D_n^*(1-\theta) \theta^3+ E_n^*\theta^4 >0, \theta \in [0,1],
\end{equation}
where\\
\begin{equation}\label{ICAFWeq26n}
\begin{split}
A_n^*=~& \lambda_n[y_n-t_n-\alpha_n(y_1-t_1)],\\
B_n^*=~& \lambda_n[2(y_n-t_n)+(y_n-t_{n+1})+h_n d_n-\alpha_n\{2(y_1-t_1)+(y_1-t_N)\\&+(x_N-x_1)d_1\}]
+[y_n-t_n-\alpha_n(y_1-t_1)],\\
C_n^*=~& \lambda_n[(y_{n+1}-t_n)-\alpha_n(y_N-t_1)+2\{(y_{n+1}-t_{n+1})-\alpha_n(y_N-t_N)\}]
+\\~&[2\{(y_n-t_n)-\alpha_n(y_1-t_1)\}+(y_n-t_{n+1})-\alpha_n(y_1-t_N)],\\
D_n^*=~&\lambda_n[y_{n+1}-t_{n+1}-\alpha_n(y_N-t_N)]+
[(y_{n+1}-t_n)+2(y_{n+1}-t_{n+1})-h_n d_{n+1}\\&-\alpha_n\{(y_N-t_1)+2(y_N-t_N)
-(x_N-x_1)d_N\}],\\
E_n^*=~& [y_{n+1}-t_{n+1}-\alpha_n(y_N-t_N)].
\end{split}
\end{equation}
With the substitution $\theta=\frac{\nu}{\nu+1}$,  (\ref{ICAFWeq26m}) is equivalent to $E_n^*\nu^4+D_n^*\nu^3 +C_n^* \nu^2 +B_n^* \nu + A_n^*>0$ for all $\nu>0$. The polynomial in (\ref{ICAFWeq26m}) is positive if $A_n^*> 0, B_n^*> 0, C_n^*> 0$, $D_n^*> 0$ and $E_n^*> 0$ are satisfied. It is straight forward to see that $A_n^*> 0$ is satisfied if $\alpha_n<\frac{y_n-t_n}{y_1-t_1}$ and $E_n^*> 0$ if $\alpha_n<\frac{y_{n+1}-t_{n+1}}{y_N-t_N}$. Assume,  $\alpha_n \le \dfrac{(y_{n+1}-t_n)+2(y_{n+1}-t_{n+1})}{(y_N-t_1)+2(y_N-t_N)}$, then the condition  $C_n^* >0$ is met if\\ $\lambda_n > \dfrac{-[2\{(y_n-t_n)-\alpha_n(y_1-t_1)\}+(y_n-t_{n+1})-\alpha_n(y_1-t_N)]}
{[(y_{n+1}-t_n)-\alpha_n(y_N-t_1)+2\{(y_{n+1}-t_{n+1})-\alpha_n(y_N-t_N)\}]}$.
It is plain to see that the additional conditions on the shape parameters $\lambda_n>0$ prescribed in the theorem ensure the positivity of  $B_n^*$ and $D_n^*$. This completes the proof.
\end{proof}

\section{Bicubic Partially Blended  Rational Fractal Interpolation Surface}\label{ICMCICMCSURFACEsec3}
We wish to a construct a  $\mathcal{C}^1$-continuous bivariate function $\Phi:  \mathbb{R} \to \mathbb{R}$ such that $\Phi(x_i,y_j)=z_{i,j}$, $\frac{\partial \Phi}{\partial x}(x_i,y_j)=z_{i,j}^x$, and $\frac{\partial \Phi}{\partial y}(x_i,y_j)=z_{i,j}^y$ for $i \in \mathbb{N}_m,  j \in \mathbb{N}_n$. This is achieved by  blending  the univariate rational quartic FIFs using the partially bicubic Coons technique \cite{Bhm}.
\subsection{Construction of rational quartic spline FIFs (Fractal boundary curves)}
Let $\Delta=\{(x_i,y_j, z_{i,j}): i \in \mathbb{N}_m,  j \in \mathbb{N}_n\}$ be a set of bivariate interpolation data, where $x_1<x_2<\dots<x_{m}$ and $y_1<y_2<\dots<y_{n}$ and  denote  $h_i=x_{i+1}-x_i,~h^*_j=y_{j+1}-y_j,~i \in \mathbb{N}_{m-1},~ j \in \mathbb{N}_{n-1}$. Set $K_{i,j}=I_i\times J_j=[x_i,x_{i+1}]\times[y_j, y_{j+1}];i \in \mathbb{N}_{m-1},~j\in \mathbb{N}_{n-1}$ be the generic subrectangular region and take $K=I\times J=[x_1,x_m]\times[y_1,y_n]$. Let $z_{i,j}^x$ and
$z_{i,j}^y$ are the $x$-partial and $y$-partial derivatives of the original function at the point $(x_i,y_j)$. Consider a surface data set $\{(x_i,y_j, z_{i,j}, z_{i,j}^x, z_{i,j}^y): i \in \mathbb{N}_m,  j \in \mathbb{N}_n\}$ placed on the rectangular grid $K$. It is plain to see that univariate data set obtained by taking sections of $K$ with the line $y=y_j$ (along the $j$-th grid line parallel to $x$-axis), $j \in \mathbb{N}_n$, namely $R_j= \{(x_i,y_j,z_{i,j},z_{i,j}^x):i \in \mathbb{N}_m\}$. Let  $L_i: I \rightarrow I_i$, be affine maps $L_i(x) = a_i x+ b_i$ satisfying $L_i(x_1)=x_i,\ L_i(x_m)=x_{i+1}$. We construct rational quartic spline FIF (fractal boundary curve):
\begin{equation}\label{ICMCSURFACEeq2}
\psi(x,y_j)= \alpha_{i,j} \psi\big( {L_i}^{-1}(x),y_j\big)+ \frac{P_{i,j}(\theta)}{Q_{i,j}(\theta)},
\end{equation}
wherein
\begin{equation*}
\begin{split}
P_{i,j}(\theta)=~& \lambda_{i,j}(z_{i,j}-\alpha_{i,j} z_{1,j})(1-\theta)^4+[\lambda_{i,j}(h_iz^x_{i,j}-\alpha_{i,j} (x_m-x_1)z^x_{1,j})]+(3\lambda_{i+1,j}(z_{i,j}-\alpha_{i,j} z_{1,j}))\cdot
\\ \nonumber
&(1-\theta)^3 \theta+[3(z_{i,j}-\alpha_{i,j} z_{1,j})+3\lambda_{i,j}(z_{i+1,j}-\alpha_{i,j} z_{m,j})](1-\theta)^2 \theta^2+[(\lambda_{i,j})(z_{i+1,j}-\\ \nonumber
&\alpha_{i,j}z_{m,j})-(h_iz^x_{i+1,j}-\alpha_{i,j}(x_m-x_1) z^x_{m,j})](1-\theta) \theta^3+ (z_{i+1,j}-\alpha_{i,j} z_{m,j})\theta^4\\
Q_{i,j}(\theta)=~&\lambda_{i,j} (1-\theta)+\theta,
\theta=~\frac{{L_i}^{-1}(x)-x_1}{x_m-x_1}=\frac{x-x_i}{h_i},~x \in I_i.
\end{split}
\end{equation*}

Similarly, for each  $i \in \mathbb{N}_m$, let us consider the univariate data set by taking sections of $K$ with the line $x=x_i$ (along the $i$-th grid line parallel to $y$-axis), namely $R_i= \{(x_i,y_j, z_{i,j},
z_{i,j}^y): j\in \mathbb{N}_n\}$. Consider the affine maps $L_j^*: [y_1,y_n] \to [y_j, y_{j+1}]$ defined by $L_j^*(y) =c_j y+ d_j$ satisfying $L_j^*(y_1)=y_j$ and $L_j^*(y_n)=y_{j+1}$, $j\in \mathbb{N}_n$.  For a fixed $i \in \mathbb{N}_m$, let $\alpha_{i,j}^*$ be the scaling  factors
along the vertical grid line $x=x_i$ such that $|\alpha_{i,j}^*|<c_j<1$. We construct rational quartic spline FIF (fractal boundary curve)::
\begin{equation}\label{ICMCSURFACEeq3}
\psi^*(x_i,y)= \alpha_{i,j}^* \psi^* \big(x_i, {L_j^*}^{-1}(y)\big)+ \frac{P_{i,j}^*(\phi)}{Q_{i,j}^*(\phi)},
\end{equation}
 where
\begin{equation*}
\begin{split}
P_{i,j}^*(\phi)=~&\lambda^*_{i,j}(z_{i,j}-\alpha^*_{i,j} z_{i,1})(1-\phi)^4+[\lambda^*_{i,j}(h^*_jz^y_{i,j}-\alpha^*_{i,j} (y_n-y_1)z^y_{i,1})]+(3\lambda^*_{i,j}(z_{i,j}-\alpha^*_{i,j} z_{i,1}))\cdot\\ \nonumber
&(1-\phi)^3 \phi+[3(z_{i,j}-\alpha^*_{i,j} z_{i,1})+3\lambda^*_{i,j}(z_{i,j+1}-\alpha^*_{i,j} z_{i,n})](1-\phi)^2 \phi^2+[(\lambda^*_{i,j})(z_{i+1,j}-\\ \nonumber
&\alpha^*_{i,j}z_{i,n})-(h^*_jz^y_{i,j+1}-\alpha^*_{i,j}(y_n-y_1)z^y_{i,n})](1-\phi) \phi^3+ (z_{i+1,j}-\alpha^*_{i,j} z_{i,n})\phi^4\\
Q_{i,j}^*(\phi)=~&\lambda^*_{i,j} (1-\phi)+\phi,
\phi=~\frac{{L_j^*}^{-1}(y)-y_1}{y_n-y_1}=\frac{y-y_j}{h_j^*},~y\in J_j.
\end{split}
\end{equation*}
\subsection{Formation of blending Rational Quartic Spline FIS}
In this section, we  blend these univariate FIFs given in (\ref{ICMCSURFACEeq2})-(\ref{ICMCSURFACEeq3}) using well-known bicubic partially  blended Coons patch  to obtain the desired surface. Consider the network of FIFs $\psi(x,y_j)$, $\psi(x,y_{j+1})$, $\psi^*(x_i,y)$, and $\psi^*(x_{i+1},y)$ for $i \in \mathbb{N}_{m-1},  j \in \mathbb{N}_{n-1}$. Consider the cubic
Hermite functions $b_{0,3}^i(x)=(1-\theta)^2(1+ 2 \theta)$, $b_{3,3}^i(x)=\theta^2 (3-2\theta)$, $b_{0,3}^j(y)=(1-\phi)^2(1+ 2 \phi)$, and $b_{3,3}^j(y)=\phi^2 (3-2\phi)$. These functions are called the blending functions, because their effect is to blend together four separate boundary curves to provide a single well-defined surface. On each individual patch $K_{i,j}= I_i \times J_j$, $i \in \mathbb{N}_{m-1}$, $j \in \mathbb{N}_{n-1}$, we define a blending rational quartic spline FIS

\begin{equation}\label{ICMCSURFACEeq4}
\begin{split}
\Phi (x,y)=&~-\begin{bmatrix}
-1 & b_{0,3}^i(x) & b_{3,3}^i(x)
\end{bmatrix}\begin{bmatrix}
0 & \psi(x,y_j) & \psi(x,y_{j+1})\\
\psi^*(x_i,y) &  z_{i,j} & z_{i,j+1}\\
\psi^*(x_{i+1},y) & z_{i+1,j} & z_{i+1,j+1}
\end{bmatrix}\begin{bmatrix}
-1\\  b_{0,3}^j(y) \\ b_{3,3}^j(y)
\end{bmatrix}
\end{split}
\end{equation}
The rational quartic spline FIS $\Phi$ can be written in the equivalent form to understand the geometry of Coons construction as follows:
\begin{equation}\label{ICMCSURFACEeq5}
\begin{split}
\Phi (x,y)=&~\begin{bmatrix}
 b_{0,3}^i(x) & b_{3,3}^i(x)
\end{bmatrix}\begin{bmatrix}
\psi^*(x_i,y) \\ \psi^*(x_{i+1},y)
\end{bmatrix}+\begin{bmatrix}
 b_{0,3}^j(y) & b_{3,3}^j(y)
\end{bmatrix}\begin{bmatrix}
\psi(x,y_j) \\ \psi(x,y_{j+1})
\end{bmatrix}\\
&~ - \begin{bmatrix}
 b_{0,3}^i(x) & b_{3,3}^i(x)
\end{bmatrix}
\begin{bmatrix}
z_{i,j} & z_{i,j+1}\\
z_{i+1,j} & z_{i+1,j+1}
\end{bmatrix}
\begin{bmatrix}
 b_{0,3}^j(y) \\ b_{3,3}^j(y)
\end{bmatrix},\\
:= &~\Phi_1 (x,y)+\Phi_2 (x,y)-\Phi_3 (x,y).
\end{split}
\end{equation}
\subsection{Bicubic partially blended rational FIS inside a rectangular parallelepiped}\label{r2pps52}
Let $\{x_i, y_j$, $z_{i,j}:i \in \mathbb{N}_m,j \in \mathbb{N}_n\}$ be an interpolation data set. For each $j \in \mathbb{N}_n$, the univariate FIF $\psi(x,y_j)$ lies inside a rectangle:
\noindent\textbf{Case-I}  We assume $0 \le \alpha_{i,j} <a_{i,j}$ for all ${i,j}\in
\mathbb{N}_{m-1}$. Then, with our assumption on $\alpha_{i,j}$,
we have $\alpha_{i,j} c \le \alpha_{i,j} z \le \alpha_{i,j} d $. This implies
$\alpha_{i,j} c+ \frac{P_{i,j}(x)}{Q_{i,j}(x)} \le \alpha_{i,j} z+ \frac{P_{i,j}(x)}{Q_{i,j}(x)} \le \alpha_{i,j} d + \frac{P_{i,j}(x)}{Q_{i,j}(x)}$. Consequently, for $G$ to lie within $R$, it suffices to
have the following conditions for all ${i,j}\in \mathbb{N}_{m-1}:$
\begin{equation}\label{ICAFWeq26a}
c \le \alpha_{i,j} c + \frac{P_{i,j}(x)}{Q_{i,j}(x)},
\end{equation}
\begin{equation}\label{ICAFWeq26b}
 \alpha_{i,j} d + \frac{P_{i,j}(x)}{Q_{i,j}(x)} \le d.
\end{equation}
Now, (\ref{ICAFWeq26a}) is satisfied if
\begin{equation}\label{ICAFWeq26c}
c(1-\alpha_{i,j}) \le \frac{A_{i,j} (1-\theta)^4 +B_{i,j}(1-\theta)^3 \theta+C_{i,j} (1-\theta)^2 \theta^2  +D_{i,j} (1-\theta) \theta^3+E_{i,j} \theta^4}{\lambda_{i,j} (1-\theta)+\theta},
\end{equation}
where the constants $A_{i,j}, B_{i,j}, C_{i,j}$, $D_{i,j}$, and $E_{i,j}$ are
given in (\ref{ICAFWeq4}). Note that the expression for
$Q_{i,j}(x)$ can be written in the degree elevated form as follows:
\begin{equation}\label{ICAFWeq26d}
\begin{split}
\lambda_{i,j} (1-\theta)+\theta \equiv~& \lambda_{i,j} (1-\theta)^3 + (2\lambda_{i,j}+1)(1-\theta)^2 \theta + (\lambda_{i,j}+2) (1-\theta) \theta^2+  \theta^3\\ \equiv~& \lambda_{i,j} (1-\theta)^4+(3\lambda_{i,j}+1)(1-\theta)^3 \theta+
(3\lambda_{i,j}+3)(1-\theta)^2 \theta^2+\\~&(\lambda_{i,j}+3) (1-\theta) \theta^3+\theta^4.
\end{split}
\end{equation}
Using (\ref{ICAFWeq26d}) in (\ref{ICAFWeq26c}), we have
\begin{equation}\label{ICAFWeq26e}
\begin{split}
&\big[A_{i,j} -c (1-\alpha_{i,j})\lambda_{i,j}\big](1-\theta)^4 + \big[B_{i,j}-c
(1-\alpha_{i,j})(3\lambda_{i,j}+1)\big]
(1-\theta)^3\theta+\big[ C_{i,j}-\\&c (1-\alpha_{i,j})(3\lambda_{i,j}+3)\big]
(1-\theta)^2 \theta^2+ \big[ D_{i,j}-c(1-\alpha_{i,j}) (\lambda_{i,j}+3) \big] (1-\theta)\theta^3+\\&\big[ E_{i,j}-c(1-\alpha_{i,j}) \big]\theta^4
\ge 0.
\end{split}
\end{equation}
With the substitution $\theta=\frac{\nu}{\nu+1}$,  (\ref{ICAFWeq26e}) is equivalent to\\
\begin{equation}\label{ICAFWeq26f}
\begin{split}
&\big[E_{i,j}-c(1-\alpha_{i,j}) \big]\nu^4+\big[ D_{i,j}-c(1-\alpha_{i,j}) (\lambda_{i,j}+3) \big]\nu^3 +\big[ C_{i,j}-c (1-\alpha_{i,j})(3\lambda_{i,j}+3)\big]\cdot\\& \nu^2 + \big[B_{i,j}-c
(1-\alpha_{i,j})(3\lambda_{i,j}+1)\big]\nu+ \big[A_{i,j} -c (1-\alpha_{i,j})\lambda_{i,j}\big]\ge0~~~ \forall ~\nu\ge0.
\end{split}
\end{equation}
It can be seen that the polynomial in (\ref{ICAFWeq26e}) is
positive if the following system of inequalities hold:
\begin{equation*}
\begin{split}
&E_{i,j}-c(1-\alpha_{i,j}) \ge 0, D_{i,j}-c(1-\alpha_{i,j}) (\lambda_{i,j}+3)\ge 0, C_{i,j}-c (1-\alpha_{i,j})(3\lambda_{i,j}+3) \ge 0,\\
&B_{i,j}-c
(1-\alpha_{i,j})(3\lambda_{i,j}+1)
\ge 0, A_{i,j} -c (1-\alpha_{i,j})\lambda_{i,j}\ge 0.
\end{split}
\end{equation*}
Since $\lambda_{i,j}>0$, the selection of $\alpha_{i,j}$ satisfying
 $\alpha_{i,j} \le
\dfrac{z_{i+1,j}-c}{z_{m,j}-c}$ and $\alpha_{i,j} \le \dfrac{z_{i,j}-c}{z_{1,j}-c}$ ensures first and fifth inequality. Let
$\alpha_{i,j} \le \min \big\{\dfrac{z_{i,j}-c}{z_{1,j}-c},
\dfrac{z_{i+1,j}-c}{z_{m,j}-c}\big\}$. By this choice of $\alpha_{i,j}$, $C_{i,j}-c (1-\alpha_{i,j})(3\lambda_{i,j}+3) \ge 0$ is obviously true because $C_{i,j}-c (1-\alpha_{i,j})(3\lambda_{i,j}+3)\equiv 3[(z_{i,j}-\alpha_{i,j}z_{1,j})-c (1-\alpha_{i,j})]+3\lambda_{i,j}[(z_{i+1,j}-\alpha_{i,j}z_{m,j})-c (1-\alpha_{i,j})]\ge 0$.
Assume,  $\alpha_{i,j} \le \dfrac{h_iz^x_{i,j}}{(x_m-x_1)z^x_{1,j}}$, then the condition $B_{i,j}-c(1-\alpha_{i,j})(3\lambda_{i,j}+1)\ge 0$ is met if $\lambda_{i,j} \ge \dfrac{-[(z_{i,j}-\alpha_{i,j}z_{1,j})-c (1-\alpha_{i,j})]}{3[(z_{i,j}-\alpha_{i,j}z_{1,j})-c (1-\alpha_{i,j})]+h_iz^x_{i,j}-\alpha_{i,j}(x_m-x_1)z^x_{1,j}}$. The condition $D_{i,j}-c(1-\alpha_{i,j}) (\lambda_{i,j}+3)\ge 0$ is met if $\lambda_{i,j} \ge -3+\dfrac{h_iz^x_{i+1,j}-\alpha_{i,j}(x_m-x_1)z^x_{m,j}}{(z_{i+1,j}-\alpha_{i,j}z_{m,j})-c (1-\alpha_{i,j})}.$ Hence, the following conditions are sufficient to verify (\ref{ICAFWeq26a}) if scaling factors and shape parameters chosen as
\begin{equation}\label{ICAFWeq26g}\left.
\begin{split}
0 ~&\le \alpha_{i,j}  < \min\Big\{\dfrac{z_{i,j}-c}{z_{1,j}-c},
\dfrac{z_{i+1,j}-c}{z_{m,j}-c},\dfrac{h_iz^x_{i,j}}{(x_m-x_1)z^x_{1,j}}\Big\},\\
\lambda_{i,j}~& \ge \lambda_{1(i,j)}:=~ \dfrac{-[(z_{i,j}-c)-\alpha_{i,j}(z_{1,j}-c) ]}{3[(z_{i,j}-c)-\alpha_{i,j}(z_{1,j}-c)]+h_iz^x_{i,j}-\alpha_{i,j}(x_m-x_1)z^x_{1,j}},\\
\lambda_{i,j}~& \ge \lambda_{2(i,j)}:=~ -3+\dfrac{h_iz^x_{i+1,j}-\alpha_{i,j}(x_m-x_1)z^x_{m,j}}{(z_{i+1,j}-c)-\alpha_{i,j}(z_{m,j}-c)}.
\end{split}\right\}
\end{equation}
Substituting the expression for $Q_{i,j}(x)$ in (\ref{ICAFWeq26d}), (\ref{ICAFWeq26b}) reduces to\\
$\big[d (1-\alpha_{i,j})\lambda_{i,j}-A_{i,j}\big](1-\theta)^4 + \big[d
(1-\alpha_{i,j})(3\lambda_{i,j}+1)-B_{i,j}\big]
(1-\theta)^3\theta+\big[d (1-\alpha_{i,j})(3\lambda_{i,j}+3)-C_{i,j}\big]
(1-\theta)^2\theta^2+ \big[d(1-\alpha_{i,j})(\lambda_{i,j}+3)-D_{i,j} \big](1-\theta) \theta^3+ \big[d(1-\alpha_{i,j})-E_{i,j} \big]\theta^4
\ge 0$.\\
Now, with the substitution $\theta=\frac{\nu}{\nu+1}$,  we may rewrite it as
$\big[d(1-\alpha_{i,j})-E_{i,j} \big]\nu^4+\big[d(1-\alpha_{i,j})(\lambda_{i,j}+3)-D_{i,j} \big]\nu^3 +\big[d (1-\alpha_{i,j})(3\lambda_{i,j}+3)-C_{i,j}\big] \nu^2 + \big[d
(1-\alpha_{i,j})(3\lambda_{i,j}+1)-B_{i,j}\big] \nu+ \big[d (1-\alpha_{i,j})\lambda_{i,j}-A_{i,j}\big]\ge0~\forall~\nu\ge0$. We proceed as above and the following conditions verify (\ref{ICAFWeq26b}):
\begin{equation}\label{ICAFWeq26h}\left.
\begin{split}
0 ~&\le \alpha_{i,j} < \min\Big\{\frac{d-z_{i,j}}{d-z_{1.j}},
\frac{d-z_{i+1,j}}{d-z_{m,j}},\frac{3(d-z_{i,j})-h_iz^x_{i,j}}{3(d-z_{1,j})-(x_m-x_1)z^x_{1,j}} \Big\},\\
\lambda_{i,j}~& \ge \lambda_{3(i,j)}:=~ \dfrac{-[(d-z_{i,j})-\alpha_{i,j}(d-z_{1,j})]}{3[(d-z_{i,j})-\alpha_{i,j}(d-z_{1.j})]-
[h_iz^x_{i,j}-\alpha_{i,j}(x_m-x_1)z^x_{1,j}]},\\
\lambda_{i,j}~& \ge \lambda_{4(i,j)}:=~ -3+\dfrac{\alpha_{i,j}(x_m-x_1)z^x_{m,j}-h_iz^x_{i+1,j}}{(d-z_{i+1,j})-\alpha_{i,j}(d-z_{m,j})}.
\end{split}\right\}
\end{equation}
Let us denote $\alpha_{i,j}^{max}=\min\Big\{a_{i,j},\dfrac{z_{i,j}-c}{z_{1,j}-c},
\dfrac{z_{i+1,j}-c}{z_{m,j}-c},\dfrac{h_iz^x_{i,j}}{(x_m-x_1)z^x_{1.j}},\dfrac{d-z_{i,j}}{d-z_{1,j}},
\dfrac{d-z_{i+1,j}}{d-z_{m,j}},\\~~~~~~~~~~~~~~~~~~~~~~~~~~~~
~~~~~~~~~~~~~~~~~~~~~\dfrac{3(d-z_{i,j})-h_iz^x_{i,j}}{3(d-z_{1,j})-(x_m-x_1)z^x_{1,j}}\Big\}$.
\\\\
\noindent\textbf{Case-II} We consider $-a_{i,j}< \alpha_{i,j} <0$. Let $(x,y) \in R$. In this case (\ref{ICAFWeq26a}) and (\ref{ICAFWeq26b}) will be replaced respectively by
\begin{equation}\label{ICAFWeq26i}
c \le \alpha_{i,j} d+ q_{i,j}(x) ~~\text{and}~~  \alpha_{i,j} c + q_{i,j}(x) \le
d.
\end{equation}
To make our considerations concise, we avoid the computational details that yield the following conditions for scaling factors and shape parameters
\begin{equation}\label{ICAFWeq26j}\left.
\begin{split}
\alpha_{i,j} ~&> \alpha_{i,j}^{min}=\max\big\{-a_{i,j}, \dfrac{z_{i,j}-c}{z_{1,j}-d},
\dfrac{z_{i+1,j}-c}{z_{m,j}-d},\dfrac{-h_iz^x_{i,j}-3
(z_{i,j}-c)}{3(d-z_{1,j})-(x_m-x_1)z^x_{1,j}},\dfrac{d-z_{i,j}}{c-z_{1,j}},\\&
~~~~~~~~~~~~~~~~~~~~~~~~~~~~~~~\dfrac{d-z_{i+1,j}}{c-z_{m,j}}, \dfrac{h_iz^x_{i,j}-3(d-z_{i,j})}{3(z_{1,j}+c)+(x_m-x_1)z^x_{1,j}}\big\}, \\
\lambda_{i,j}~& \ge \lambda_{5(i,j)}:=~ \dfrac{-[\alpha_{i,j}(d-z_{1,j})-(c-z_{i,j})]}{3[\alpha_{i,j}(d-z_{1,j})-(c-z_{i,j})]+
[h_iz^x_{i,j}-\alpha_{i,j}(x_m-x_1)z^x_{1,j}]},\\
\lambda_{i,j}~& \ge \lambda_{6(i,j)}:=~ -3+\dfrac{h_iz^x_{i+1,j}-\alpha_{i,j}(x_m-x_1)z^x_{m,j}}{\alpha_{i,j}(d-z_{m,j})-(c-z_{i+1,j})},\\
\lambda_{i,j}~& \ge \lambda_{7(i,j)}:=~ \dfrac{-[(d-z_{i,j})-\alpha_{i,j}(c-z_{1,j})]}{3[(d-z_{i,j})-\alpha_{i,j}(c-z_{1,j})]-
[h_id^x_{i,j}-\alpha_{i,j}(x_m-x_1)z^x_{1,j}]},\\
\lambda_{i,j}~& \ge \lambda_{8(i,j)}:=~ -3+\dfrac{\alpha_{i,j}(x_m-x_1)z^x_{m,j}-h_iz^x_{i+1,j}}{(d-z_{i+1,j})-\alpha_{i,j}(c-z_{m,j})},
\end{split}\right\}
\end{equation}
Similarly, we can find the condition for other fractal boundary curve and blend them to get the final result.
\subsection{Bicubic partially blended rational FIS above a prescribed plane}
\begin{theorem}
Let $\{x_i, y_j$, $z_{i,j}:i \in \mathbb{N}_m,j \in \mathbb{N}_n\}$ be an interpolation data set. For each $j \in \mathbb{N}_n$, the univariate FIF $\psi(x,y_j)$ lies above the line  $t=c[1-\frac{x}{a}-\frac{y_j}{b}]$ if the scaling factors and the shape parameters are selected so as to satisfy the following conditions
\begin{equation*}
\begin{split}
0\le \alpha_{i,j} < \min \big\{\dfrac{z_{i,j}-t_{i,j}}{z_{1,j}-t_{1,j}},
\dfrac{z_{i+1,j}-t_{i+1,j}}{z_{m,j}-t_{m,j}}\big\} ,~\lambda_{i,j} > \max \{0,\lambda_{9(i,j)}, \lambda_{10(i,j)}\},
\end{split}
\end{equation*}
where
$\lambda_{i,j} \ge \lambda_{9(i,j)}:=\dfrac{-[2\{(z_{i,j}-t_{i,j})-\alpha_{i,j}(z_{1,j}t_{1,j})\}+(z_{i,j}-t_{i+1,j})-\alpha_{i,j}(z_{1,j}-t_{m,j})]}
{[(z_{i+1,j}-t_{i,j})-\alpha_{i,j}(z_{m,j}-t_{1,j})+2\{(z_{i+1,j}-t_{i+1,j})-\alpha_{i,j}(z_{m,j}-t_{m,j})\}]}$.
$\lambda_{i,j} \ge \lambda_{10(i,j)}:= -\dfrac{z_{i,j}-t_{i,j}-\alpha_{i,j}(z_{1,j}-t_{1,j})}{l_{i,j}},\\ \lambda_{i,j} \ge
\lambda_{11(i,j)}:= -\dfrac{m_{i,j}}{[z_{i+1,j}-t_{i+1,j}-\alpha_{i,j}(z_{m,j}-t_{m,j})]}$,
$l_{i,j}=[2(z_{i,j}-t_{i,j})+(z_{i,j}-t_{i+1,j})+h_i z^x_{i,j}-\alpha_{i,j}\{2(z_{1,j}-t_{1,j})+(z_{1,j}-t_{m,j})+(x_m-x_1)z^x_{1,j}\}]$,
$m_{i,j}=[(z_{i+1,j}-t_{i,j})+2(z_{i+1,j}-t_{i+1,j})-h_i z^x_{i+1,j}-\alpha_{i,j}\{(z_{m,j}-t_{1,j})+2(z_{m,j}-t_{m,j})
-(x_m-x_1)z^x_{m,j}\}]$.
\end{theorem}
\begin{proof}
Let $\{x_i, y_j,z_{i,j}: i \in \mathbb{N}_m,j \in \mathbb{N}_n \}$ be an interpolation data set lying above the plane  $t=c[1-\frac{x}{a}-\frac{y}{b}]$,
i,e. $z_{i,j}>t_{i,j}=c[1-\frac{x}{a}-\frac{y}{b}]$ for all $i \in \mathbb{N}_m,j \in \mathbb{N}_n$. We wish to find conditions on the parameters of rational quartic FIS so that it lies above the aforementioned plane, that is $\Phi(x,y)>t(x,y)$ for all $(x,y) \in K$. We recall that the surface generated by the rational quartic FIS $\psi$ lies above the plane if the network of boundary curves $\psi(x,y_j)$ $\forall$ $j \in \mathbb{N}_n$ and $\psi^*(x_i,y)$ $\forall$ $i \in \mathbb{N}_m$ lie above the plane. Since, $\Phi(x_i,y_j)= z_{i,j}>t(x_i,y_j)$ for all $i \in \mathbb{N}_m,j \in \mathbb{N}_n$, the proof of $\psi(\tau,y_j)> t(\tau,y_j)$ for all $\tau \in I$ is equivalent to find the conditions for which
$\psi(x,y_j)> t(x,y_j)$, $x \in I$ implies $\psi(L_i(x),y_j)> t(L_i(x),y_j)$ for $x \in I$. Assume $\psi(x,y_j)> t(x,y_j)$. We need to prove that
\begin{equation}\label{ICMCSURFACEeq6}
\alpha_{i,j} \psi\big(x,y_j\big)+ \frac{P_{i,j}(\theta)}{Q_{i,j}(\theta)}>c[1-\frac{a_i x+ b_i}{a}-\frac{y_j}{b}],
\end{equation}
Since $Q_{i,j}(\theta)>0$, in view of assumptions $\psi(x,y_j)> t(x,y_j)$ and $\alpha_{i,j} \ge 0$ for all $i \in \mathbb{N}_m,j \in \mathbb{N}_n$, we deduce that the following condition confirm (\ref{ICMCSURFACEeq6}).
\begin{equation}\label{ICMCSURFACEeq7}
\alpha_{i,j}c[1-\frac{x}{a}-\frac{y_j}{b}]Q_{i,j}(\theta)+ P_{i,j}(\theta) -\{c[(1-\frac{y_j}{b})-\frac{a_i x+ b_i}{a}]\}Q_{i,j}(\theta)>0.
\end{equation}
Substituting $x=x_1+\theta(x_m-x_1)$, the expression for $P_{i,j}(\theta)$ and using the degree elevated form of $Q_{i,j}(\theta)$ from (\ref{ICAFWeq26d}), (\ref{ICAFWeq26l}) reduces to
\begin{equation}\label{ICAFWeq26m}
A_{i,j}^*(1-\theta)^4 +B_{i,j}^*(1-\theta)^3 \theta
+ C_{i,j}^*(1-\theta)^2\theta^2 + D_{i,j}^*(1-\theta) \theta^3+ E_{i,j}^*\theta^4 >0, \theta \in [0,1],
\end{equation}
where\\
\begin{equation}\label{ICAFWeq26n}
\begin{split}
A_{i,j}^*=~& \lambda_{i,j}[z_{i,j}-t_{i,j}-\alpha_{i,j}(z_{1,j}-t_{1,j})],\\
B_{i,j}^*=~& \lambda_{i,j}[2(z_{i,j}-t_{i,j})+(z_{i,j}-t_{i+1,j})+h_i z^x_{i,j}-\alpha_{i,j}\{2(z_{1,j}-t_{1,j})+(z_{1,j}-t_{m,j})\\&+(x_m-x_1)z^x_{1,j}\}]
+[z_{i,j}-t_{i,j}-\alpha_{i,j}(z_{1,j}-t_{1,j})],\\
C_{i,j}^*=~& \lambda_{i,j}[(z_{i+1,j}-t_{i,j})-\alpha_{i,j}(z_{m,j}-t_{1,j})+2\{(z_{i+1,j}-t_{i+1,j})-\alpha_{i,j}(z_{m,j}-t_{m,j})\}]
+\\~&[2\{(z_{i,j}-t_{i,j})-\alpha_{i,j}(z_{1,j}-t_{1,j})\}+(z_{i,j}-t_{i+1,j})-\alpha_{i,j}(z_{1,j}-t_{m,j})],\\
D_{i,j}^*=~&\lambda_{i,j}[z_{i+1,j}-t_{i+1,j}-\alpha_{i,j}(z_{m,j}-t_{m,j})]+
[(z_{i+1,j}-t_{i,j})+2(z_{i+1,j}-t_{i+1,j})-h_i z^x_{i+1,j}\\&-\alpha_{i,j}\{(z_{m,j}-t_{1,j})+2(z_{m,j}-t_{m,j})
-(x_m-x_1)z^x_{m,j}\}],\\
E_{i,j}^*=~& [z_{i+1,j}-t_{i+1,j}-\alpha_{i,j}(z_{m,j}-t_{m,j})].
\end{split}
\end{equation}
With the substitution $\theta=\frac{\nu}{\nu+1}$,  (\ref{ICAFWeq26m}) is equivalent to $E_{i,j}^*\nu^4+D_{i,j}^*\nu^3 +C_{i,j}^* \nu^2 +B_{i,j}^* \nu + A_{i,j}^*>0$ for all $\nu>0$. The polynomial in (\ref{ICAFWeq26m}) is positive if $A_{i,j}^*> 0, B_{i,j}^*> 0, C_{i,j}^*> 0$, $D_{i,j}^*> 0$ and $E_{i,j}^*> 0$ are satisfied. It is straight forward to see that $A_{i,j}^*> 0$ is satisfied if $\alpha_{i,j}<\frac{z_{i,j}-t_{i,j}}{z_{1,j}-t_{1,j}}$ and $E_{i,j}^*> 0$ if $\alpha_{i,j}<\frac{z_{i+1,j}-t_{i+1,j}}{z_{m,j}-t_{m,j}}$. Assume,  $\alpha_{i,j} \le \dfrac{(z_{i+1,j}-t_{i,j})+2(z_{i+1,j}-t_{i+1,j})}{(z_{m,j}-t_{1,j})+2(z_{m,j}-t_{m,j})}$, then the condition  $C_{i,j}^* >0$ is met if\\ $\lambda_{i,j} > \dfrac{-[2\{(z_{i,j}-t_{i,j})-\alpha_{i,j}(z_{1,j}-t_{1,j})\}+(z_{i,j}-t_{i+1,j})-\alpha_{i,j}(z_{1,j}-t_{m,j})]}
{[(z_{i+1,j}-t_{i,j})-\alpha_{i,j}(z_{m,j}-t_{1,j})+2\{(z_{i+1,j}-t_{i+1,j})-\alpha_{i,j}(z_{m,j}-t_{m,j})\}]}$.
It is plain to see that the additional conditions on the shape parameters $\lambda_{i,j}>0$ prescribed in the theorem ensure the positivity of  $B_{i,j}^*$ and $D_{i,j}^*$. This completes the proof.
\end{proof}
Similarly, we can find the condition for other fractal boundary curve.
\begin{theorem}\label{main}
Let $\{x_i, y_j$, $z_{i,j}:i \in \mathbb{N}_m,j \in \mathbb{N}_n \}$ be an interpolation data set that lies above the plane $t=c[1-\frac{x}{a}-\frac{y}{b}]$ i.e. $z_{i,j}>t_{i,j}$  for all $i \in \mathbb{N}_m,j \in \mathbb{N}_n$. Then the
rational quartic spline FIS $\Phi$ (cf. (\ref{ICMCSURFACEeq4})) lies above the plane provided the horizontal scaling parameters $\alpha_{i,j}$ for $i \in \mathbb{N}_{m-1},j \in \mathbb{N}_n$ and the vertical scaling parameters $\alpha^*_{i,j}$ for $i \in \mathbb{N}_m,j \in \mathbb{N}_{n-1}$, the horizontal shape parameters $\lambda_{i,j}$,for $i \in \mathbb{N}_{m-1},j \in \mathbb{N}_n$ and the vertical shape parameters $\lambda^*_{i,j}$ for
$i \in \mathbb{N}_m,j \in \mathbb{N}_{n-1}$ satisfy the hypotheses of the above Theorem.
\end{theorem}

\section{Convergence Analysis for $\alpha$-Fractal Rational Quartic Spline}
\begin{theorem}\label{ICAFWthm1}
Let $\Phi$ be the bicubic partially blended rational FIS (cf. (\ref{ICMCSURFACEeq4})) corresponding to
the bivariate data $\{x_i, y_j$, $z_{i,j}:i \in \mathbb{N}_m,j \in \mathbb{N}_n \}$ generated from an original function $H\in C^1(R)$. Then, $\Phi$ converges uniformly to $H$ as $h^\dag \to 0$, where $h^\dag:= \max\{h, h^*\}$.
\end{theorem}
\begin{proof}
let  $\Phi$ and $C$, respectively, be the bicubic partially blended rational FIS and its traditional counterpart.
We will get our desired result by using the triangle inequality:
$\|\Phi-H\|_{\infty} \le \|\Phi - C\|_{\infty} + \|C - H\|_{\infty}$.
\begin{equation}\label{convergence1}
\begin{split}
|\Phi(x,y)-C(x,y)| \le ~& b_{0,3}^i(x)|\psi^*(x_i,y,\alpha_{i,j}^*)-\psi^*(x_i,y,0)|+b_{3,3}^i(x)|\psi^*(x_{i+1},y,\alpha_{i+1,j}^*)-\psi^*(x_i,y,0)|,\\
+ ~& b_{0,3}^j(y)|\psi(x,y_j,\alpha_{i,j})-\psi(x,y_j,0)|+b_{3,3}^j(y)|\psi(x,y_{j+1},\alpha_{i,j+1})-\psi(x,y_{j+1},0)|.
\end{split}
\end{equation}
We use here the proposition 4.2 appeared in \cite{KCFractal2} to get the following,\\ 
$|\psi^*(x_i,y,\alpha_{i,j}^*)-\psi^*(x_i,y,0)|\le \frac{h^*}{|J|-h^*}(\max\{z_{i,j}:j \in \mathbb{N}_n\}+\frac{1}{4}h^*\max\{z^y_{i,j}:j \in \mathbb{N}_n\}+\max\{|z_{i,1}|,|z_{i,n}|\}+\frac{1}{4}|J|\max\{|z^y_{i,1}|,|z^y_{i,n}|\})$,\\

$|\psi(x,y_j,\alpha_{i,j})-\psi(x,y_j,0)|\le \frac{h}{|I|-h}(\max\{z_{i,j}:i \in \mathbb{N}_m\}+\frac{1}{4}h\max\{z^x_{i,j}:i \in \mathbb{N}_m\}+\max\{|z_{1,j}|,|z_{m,j}|\}+\frac{1}{4}|I|\max\{|z^x_{1,j}|,|z^x_{n,j}|\})$,
We know that 
$b_{0,3}^i(x)+b_{3,3}^i(x)=b_{0,3}^j(y)+b_{3,3}^j(y)=1$, we get from (\ref{convergence1}):

\begin{equation*}\label{convergence1a}
\begin{split}
|\Phi(x,y)-C(x,y)| \le ~& \frac{h^*}{|J|-h^*}\big\{|Z|_{\infty}+\frac{1}{4}h^*|Z^y|_{\infty}+|Z^*_e|_{\infty}+\frac{1}{4}|J||D^*_e|_{\infty}\big\}
+\frac{h}{|I|-h}\big\{|Z|_{\infty}+\frac{1}{4}h|Z^x|_{\infty}\\&+|Z_e|_{\infty}+\frac{1}{4}|I||D_e|_{\infty}\big\},
\end{split}
\end{equation*}
It is true for all values,
\begin{equation}\label{convergence1}
\begin{split}
||\Phi(x,y)-C(x,y)||_{\infty} \le ~& \frac{h^*}{|J|-h^*}\big\{|Z|_{\infty}+\frac{1}{4}h^*|Z^y|_{\infty}+|Z^*_e|_{\infty}+\frac{1}{4}|J||D^*_e|_{\infty}\big\}
+\frac{h}{|I|-h}\big\{|Z|_{\infty}+\frac{1}{4}h|Z^x|_{\infty}\\&+|Z_e|_{\infty}+\frac{1}{4}|I||D_e|_{\infty}\big\}.
\end{split}
\end{equation}
The above triangle inequality asserts that
$\|\Phi-H\|_{\infty}\to 0$ as $h$ and $h^*$ tend to zero.
\end{proof}


\begin{thebibliography}{9}
\bibitem{Bhm} {B\"{o}hm, W.}: A survey of curve and surface methods in CAGD, Computer Aided Geom. Des. \textbf{1}, 1-60 (1984).
\bibitem{KK} S. A. A. Karim, V. P. Kong and  I. Hashim, Positivity preserving using $GC^1$ rational quartic spline, \emph{AIP Conf. Proc.} \textbf{1482} (2012) 26--31.
\bibitem{ABD} A. R. M. Piah and K. Unsworth, Improved sufficient conditions for monotonic piecewise rational quartic interpolation, \emph{Sains Malaysian} \textbf{40}(10) (2011) 1173--1178.
\bibitem{WQ} Q. Wang and J. Tan, Rational quartic spline involving shape parameters, \emph{J. Comp. Information Systems.}  \textbf{1}(1) (2004) 127--130.
\bibitem{B1}  M. F. Barnsley, Fractal functions and interpolation, \emph{Constr. Approx.} \textbf{2}(4) (1986) 303--329.
\bibitem{CK6} A. K. B. Chand,  S. K. Katiyar and P. Viswanathan, Approximation using hidden variable fractal interpolation functions, \emph{J. Fractal Geom.} \textbf{2}(1) (2015) 81--114.
\bibitem{CNVS} A. K. B. Chand, M. A. Navascu\'{e}s,  P. Viswanathan and S. K. Katiyar, Fractal trigonometric polynomial for restricted range approximation, \emph{Fractals} \textbf{24}(2) (2016) 1--11.
\bibitem{CK7} S. K. Katiyar, A. K. B. Chand and M. A. Navascu\'{e}s, Hidden Variable $\mathbf{A}$-Fractal Functions and Their Monotonicity Aspects, \emph{Rev. R. Acad. Cienc. Zaragoza} \textbf{71} (2016) 7-30.
\bibitem{NS3} M. A. Navascu\'{e}s and M. V. Sebasti\'{a}n, Smooth fractal interpolation, \emph{J. Inequal. Appl. Article ID 78734}, (2004) 20 pp.
\bibitem{N1} M. A. Navascu\'{e}s, Fractal polynomial interpolation, \emph{Z. Anal. Anwend.} \textbf{25}(2) (2005) 401--418.
\bibitem{NVCSS} M. A. Navascu\'{e}s, P. Viswanathan, A. K. B. Chand, M. V. Sebasti\'{a}n, and S. K. Katiyar, Fractal bases for Banach spaces of smooth functions, \emph{Bull. Aust. Math. Society} \textbf{92} (2015) 405--419.
\bibitem{AKBCSK1} A. K. B. Chand, S. K. Katiyar and Saravana Kumar G., A new class of rational fractal function for curve
fitting, \emph{Proceeding of Computer Aided Engineering CAE 2013}, ISBN No. 78-93-80689-17-3.
\bibitem{CKIITR} A. K. B. Chand and S. K. Katiyar, Quintic Hermite fractal interpolation in a strip: preserving copositivity, \emph{Springer Proc. Math. Stat.} \textbf{143} (2015) 463--475.
\bibitem{KC} S. K. Katiyar and A. K. B. Chand, Toward a unified methodology for fractal extension of various shape preserving spline interpolants, \emph{Springer Proc. Math. Stat.,} \textbf{139} (2015) 223--238.
\bibitem{SKKKM} S. K. Katiyar, K. M. Reddy and A. K. B. Chand, Constrained data visualization using rational bi-cubic
fractal functions, \emph{Comp. Inform. Science Springer} (2017) 321--330.
\bibitem{SKKthesis}  S. K. Katiyar, Shape Preserving Rational and Coalescence Fractal Interpolation Functions and Approximation by Variable Scaling Fractal Functions. Ph. D. thesis, \emph{Indian Institute of Technology Madras, India,} (2017).
\bibitem{KCGS} S. K. Katiyar, A. K. B. Chand and Saravana Kumar G., A new class of rational cubic spline fractal interpolation function and its constrained aspects, \emph{Appl. Math. Comp.} \textbf{346} (2019) 319-335.
\bibitem{KCSJha} S. K. Katiyar, A. K. B. Chand and Sangita Jha, Parameter Identification of Constrained Data by a New Class
of Rational Fractal Function, \emph{arXiv:1809.08206}.
\bibitem{KCSole}   S. K. Katiyar and A. K. B. Chand, A new class of monotone/convex rational fractal function, \emph{arXiv:1809.10682}.

\bibitem{KCFractal2} S. K. Katiyar and A. K. B. Chand, Shape Preserving Rational Quartic Fractal Functions, Fractal, in Press.
\bibitem{B}  M. F. Barnsley, Fractals Everywhere. \emph{Academic Press}, Dublin (1988). (2nd Edition, Morgan Kaufmann
1993; 3rd Edition, Dover Publications, 2012).
\bibitem{PRM} P. R. Massopust, Fractal Functions, Fractal Surfaces and Wavelets, \emph{Academic Press,} 1994.
\end{thebibliography}
\end{document}